\documentclass[12pt,a4paper]{article}

\usepackage{theorem,enumerate}
\usepackage{amsmath,latexsym,amssymb,amsfonts, epsfig}
\usepackage{eucal}
\usepackage{mathrsfs}
\usepackage{array}
\usepackage{epstopdf} 
\usepackage{MnSymbol}

\newcounter{Scounter}
\theorembodyfont{\normalfont\slshape}
\newtheorem{thm}{Theorem}

\newtheorem{cor}[thm]{Corollary}

\newtheorem{definition}[thm]{Definition}
\newtheorem{lemma}[thm]{Lemma}

\newtheorem{con}[thm]{Conjecture}
\newtheorem{prob}{Problem}

\newcommand{\PROOF}[1]{\medbreak\noindent\textbf{Proof #1.}\ }

\newcommand{\qed}{{$\quad\square$\vs{3.6}}}

\newcommand{\vs}[1]{\vspace*{#1 mm}}

\bfseries\normalfont

\makeatletter
\def\thanks#1{%
   \footnotemark
   \edef\@tempa{\noexpand\noexpand\noexpand\footnotetext[\the\c@footnote]}%
   \toks@\expandafter{\@thanks}%
   \toks\tw@{{#1}}
   \xdef\@thanks{\the\toks@\@tempa\the\toks\tw@}}
\makeatother

\begin{document}

\title{Cycle Double Covers via Kotzig Graphs}

\author{
Herbert Fleischner, Roland H\"aggkvist, Arthur Hoffmann-Ostenhof
}
\date{}
\maketitle

\begin{abstract}
We show that every $2$-connected cubic graph $G$ has a cycle double cover if $G$ has a spanning subgraph $F$ such that (i) every component of $F$ has an even number of vertices (ii) every component of $F$ is either a cycle or a subdivision of a Kotzig graph and (iii) the components of $F$ are connected to each other in a certain general manner.
\end{abstract}

\noindent
{\bf Keywords:} 
cubic graph, 3-regular graph, cycle double cover, frame, Kotzig graph.

\section{Introduction and definitions}

All graphs here are finite and considered loopless if not stated otherwise. 
A \textit {cycle} is a $2$-regular connected graph. An \textit {eulerian graph} is a (not necessarily connected) graph where every vertex has even degree. The disjoint union is denoted by $\cupdot$. A $vw$-path is a path with endvertices $v$ and $w$. An edge-colored cycle is called \textit {bicolored} if the edges of the cycle are colored with at most two colors. Within a proof "if and only if" may be abbreviated by "iff". For terminology not defined in this paper, we refer to \cite{BM}. 

A \textit {Kotzig graph} is a cubic graph admitting a proper $3$-edge-coloring such that each bicolored cycle is hamiltonian.
Such edge-coloring of a cubic graph is called a \textit{Kotzig-coloring}.
The $\theta$-graph is the unique bridgless cubic graph with two vertices. It is the smallest Kotzig graph.

A \textit{cycle cover} of a graph $G$ is a set of cycles $S$ of $G$ such that each edge of $G$ is contained in at least one cycle of $S$.
A \textit{$k$-cycle cover} is a cycle cover $S$ which has a partition into $k$ subsets such that each consists of pairwise edge disjoint cycles.
A \textit {cycle double cover} (CDC) $S$ of a graph $G$ is a set of cycles of $G$ such that every edge of $G$ is contained in precisely two cycles of $S$. 
A CDC $S$ of a cubic graph is called a $k$-CDC if $S$ can be partitioned into $k$ subsets such that each consists of pairwise edge-disjoint cycles.

The well known \textit {Cycle Double Cover Conjecture}(CDCC) states that every $2$-connected cubic graph has a CDC.
 There are several stronger versions of the CDCC. For a survey on the CDCC, see \cite{Z}. 
 
Let $H$ be a graph, then 
a \textit {subdivision} of $H$ is a graph $H'$ which is constructed from $H$ by replacing at least one edge $e \in E(H)$, say $e=vw$, by a $vw$-path of length at least $2$ (denoted by $e'$ in $H'$). \\
The degree of a vertex $v$ in a graph $G$ is denoted by $d(v,G)$. The set of edges incident with $v$ is 
denoted by $E_v$. Vertex-colorings and edge-colorings need not be proper colorings, i.e., vertices, edges respectively of the same color may be adjacent. The number of edges with color $c$ incident with $v \in V(G)$ is denoted by $d_c(v,G)$. Every $3$-coloring of a vertex or edge set of a graph uses colors from the set $\{1,2,3\}$. 
A $3$-total-coloring of a graph $G$ consists of a $3$-edge-coloring and a $3$-vertex-coloring of $G$.

\begin{definition}\label{kotzig}
Let $G$ be a $2$-connected cubic graph.\\
{\bf (a)} A Kotzig-frame of $G$ is a spanning subgraph $F$ of $G$
such that \\
(i) every component of $F$ has an even number of vertices and \\
(ii) every component of $F$ is either a cycle or a subdivision of a Kotzig graph.\\
{\bf (b)} We call a component of $F$ a C-component if it is a cycle; otherwise we call it a K-component.\\
{\bf (c)} An edge $e$ with $e \not\in E(F)$ is called a chord of $F$ if both endvertices of $e$ are in the same component of $F$. The set of chords of $F$ is denoted by $F^*$. \\
{\bf (d)} Moreover, $G_F$ denotes the eulerian graph obtained by contracting every component of $F$ to a distinct vertex and deleting all loops which arise in the contracting process. A vertex of $G_F$ is called a C-vertex if it corresponds to a C-component, otherwise it is called a K-vertex.
\end{definition}

It is known that every cubic graph $G$ with a Kotzig-frame $F$ has a CDC if the number of K-components in $F$ is at most one, see \cite{HM1}. Other spanning subgraphs than Kotzig graphs and even cycles also imply the existence of a CDC. For more information regarding this approach towards a solution of the CDCC, we refer to \cite{FH,HM2,YZ}. Note that not every $3$-connected cubic graph has a Kotzig-frame, see \cite{HO}. However, so far there is no cyclically $4$-edge connected cubic graph known without a Kotzig-frame.

\begin{con}\label{c:4}
Every cyclically $4$-edge connected cubic graph has a Kotzig-frame.
\end{con}


\begin{prob}
Does every cyclically $4$-edge connected cubic graph have a Kotzig-frame $F$ such that every component of $F$ is either a cycle or a subdivision of a $\theta$-graph?
\end{prob}

In order to state the main theorem we need some "color terminology" concerning Kotzig graphs.

\begin{definition}\label{perfect}
Let $\alpha$ be a $3$-total-coloring of a subdivision $K'$ of a Kotzig graph $K$.
Then $\alpha$ is said to be a perfect coloring of $K'$ if\\
(i) two edges in $K'$ have the same edge color if they are incident with the same $2$-valent vertex,\\
(ii) each $2$-valent vertex of $K'$ has the same color as its incident edges,\\
(iii) the implied $3$-edge-coloring of $K$ (every edge $e \in E(K)$ has the same color as any edge of the corresponding path $e'$ in $K'$) is a Kotzig-coloring of $K$. 
\end{definition}

Note that the vertex color of vertices $v$ with $d(v,K')=3$ is not relevant for this definition and also not for later purposes. Call a $3$-total-coloring $\alpha$ of a Kotzig-frame $F$ of a cubic graph $G$ \textit{perfect} if the coloring restricted to every K-component of $F$ is perfect and if in each C-component all vertices and edges have the same color. Two vertices in different C-components may have different colors. For an example of a perfect coloring, see Figure 1.

\begin{definition}\label{wellconnected}
Let $F$ be a Kotzig-frame of a cubic graph $G$ and $\alpha$ a perfect coloring of $F$.
Let $G_{F,\alpha}$ be the partially $3$-edge colored graph $G_F$ which results 
from coloring every edge $e$ of $G_F$ which corresponds to an edge $e=vw$ of $G-E(F)-F^*$ satisfying $\alpha(v) = \alpha (w)$, with color $\alpha (v)$; all other edges of $G_{F,\alpha}$ are uncolored.
Then $G_F$ is said to be well connected (with respect to $\alpha$) if $G_{F,\alpha}$ has a connected subgraph $H$ 
such that $H$ contains all K-vertices and where all edges of $H$ 
have the same color. If such perfect coloring of $F$ exists, then we say for brevity's sake that $F$ is well connected.

\end{definition}


We formulate the main result.

\begin{thm}\label{main}
Let $G$ be a cubic graph with a well connected Kotzig-frame, then $G$ has a $6$-cycle double cover.
\end{thm}

Since it is straightforward to see that the Kotzig-frames in the next two corollaries are well connected, 
we obtain by the above theorem the following results.

\begin{cor}\label{mainimply}
Let $G$ be a cubic graph with a Kotzig-frame $F$. Let $\tilde K$ be the set of K-vertices and 
$\tilde C$ the set of C-vertices of $G_F$. Then $G$ has a $6$-cycle double cover if one of the following holds.\\
(i) $\tilde K$ is an independent set in $G_F$ and $G_F-\tilde K$ is connected.\\
(ii) every vertex of $\tilde K$ is adjacent to a vertex of $\tilde C$ and $G_F-\tilde K$ is connected.\\ 
(iii) there is a subset $C^* \subseteq \tilde C$ such that $G_F[C^*]$ is connected and $\tilde K \subseteq N(C^*)$.
\end{cor}



\begin{cor}\label{mainimply1}
Let $G$ be a cubic graph with a Kotzig-frame $F$ which has at most two K-components, then $G$ has a $6$-cycle double cover.
\end{cor}

\begin{prob}
Has every cyclically $4$-edge connected cubic graph a well connected Kotzig-frame?
\end{prob}

To prove Theorem \ref{main}, we will transform $G$ in the next sections into a so called $3$-row graph $R$ (see 
Def.\ref{rowgraph}) and show that $R$ has a certain edge and vertex coloring (see Def.\ref{d:swedish}) which will imply by Theorem \ref {t:swedish} a $6$-CDC of $G$. This approach could also lead to a proof that every cubic graph with a Kotzig-frame has a CDC by solving Conjecture \ref{c:row} and applying Theorem \ref {t:swedish} (see below).

\begin{definition}\label{rowgraph}
A $3$-row graph $R$ is a graph with vertices $v_{ij}$ where $1 \leq i \leq 3$, $1 \leq j \leq s$ such that 
$C_j$ is an independent set of $R$ where $C_j:=\{v_{1j},v_{2j},v_{3j}\}$ denotes the vertex set of the jth-column of $R$. Moreover, $V_i=\{v_{ij}:1\leq j \leq s\} $ denotes the vertex set of the $i$-th row of $R$. $R_C$ with $V(R_C)=\{c_1,c_2,...,c_s\}$ is the graph which is obtained from $R \cupdot \{c_1,c_2,...,c_s\} $ by identifying all vertices of $C_j$ with the vertex $c_j$, $j=1,2,...,s$. 
\end{definition}

Finally, we introduce a special $3$-total-coloring for 3-row graphs which we call \textit{amiable coloring}.

\begin{definition}\label{d:swedish}
An amiable coloring of a $3$-row graph $R$ is a $3$-vertex-coloring $f$ together with a $3$-edge-coloring $g$ such that the following holds for all $i \in \{1,2,3\}$ and all $j \in \{1,2,...,s\}$:\\
(i) $f(v) \not=f(w) \,\text{for every} \,\,\{v,w\} \subseteq C_j$ with $v \not= w$, \\
(ii) $f(v) \not= g(e) \,\text{for every} \, \,v \in V(R) \,\,\text{and for every} \,\,e \in E_v$, \\
(iii) $d_i(v_{1j}, R)+d_i(v_{2j}, R)+d_i(v_{3j}, R)$ is even.
\end{definition}

We use the ordered pair $(f,g)$ to denote an amiable coloring.

\begin{con}\label{c:row}
Let $R$ be a $3$-row graph such that $R_C$ is eulerian, then $R$ admits an amiable coloring.
\end{con}

By Theorem \ref{t:swedish} and since a minimum counterexample to the CDCC is cyclically $4$-edge connected, we obtain

\begin{cor} \label{c:big}
The truth of Conjecture \ref{c:row} and Conjecture \ref{c:4} proves the CDCC.
\end{cor}

Conjecture \ref{c:row} is proven within this paper for several classes of $3$-row graphs, see Corollaries \ref{cor:rows}, \ref{cor:rows2} and \ref{cor:rows3}. Note that Corollary \ref{cor:rows2} was the original motivation for writing this paper.

\section{$3$-row graphs and amiable colorings}

If we rename in a $3$-row graph $R$, $v_{ip}$ by $v_{iq}$ and $v_{iq}$ 
by $v_{ip}$ for $i=1,2,3$, then we say that the $p$-th column and the $q$-th column in $R$ are permuted. Similarly we say that we permute vertices within a column if the names of the vertices in that column are permuted.
A \textit {rearrangement} of a $3$-row graph $R$, is a $3$-row graph $R'$ which is obtained from $R$ 
by permuting columns of $R$ and by permuting vertices in individual columns. 
Note that $R$ and $R'$ are isomorphic graphs. 
The proof of the following lemma is straightforward by retaining the edge-coloring during the process of rearranging $R$.

\begin{lemma}\label{arrange}
Let $R'$ be a rearrangement of a $3$-row graph $R$, then $R'$ has an amiable coloring if and only if $R$ has an amiable coloring. 
\end{lemma}

The next definition is essential in the paper.

\begin{definition}\label{d:detachment}
Let $G$ be a connected cubic graph with a Kotzig-frame $F$ with $s$ components $L_1,L_2,...,L_s$ and a perfect coloring $\alpha$ of $F$.
Then the $3$-row graph $R(G,F,\alpha)$ is obtained by identifying in $G \cupdot \{v_{ij}:1\leq i\leq 3, \, 1\leq j\leq s\}$ all $2$-valent vertices of $ \{v \in L_j: \,\alpha (v)=i\}$ with $v_{ij}$, by deleting all $3$-valent vertices of the K-components of $F$, and by deleting all resulting loops.
\end{definition}

Figure \ref{f:row} exhibits the $3$-row graph $R(G_0,F_0,\alpha_0)$ of the graph $G_0$ of Figure \ref{f:perfect}. Note that in $R(G,F,\alpha)$, a column representing a C-component of $F$ contains precisely two isolated vertices.

\begin{figure}[htpb]\label{f:perfect}
\centering\epsfig{file= 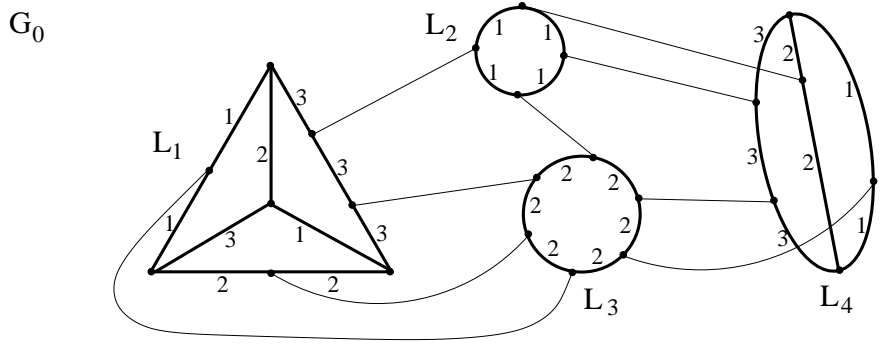,width=4.5 in}
\caption{A perfect coloring $\alpha_0$ (only the edge colors are illustrated) of a Kotzig-frame $F_0$ of a cubic graph 
$G_0$. The components $L_1,L_2,L_3,L_4$ of $F_0$ are in boldface.} 
\end{figure}

\begin{figure}[htpb]\label{f:row}
\centering\epsfig{file= 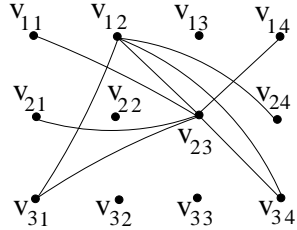,width=1.5 in}
\caption{The graph $R(G_0,F_0,\alpha_0)$.} 
\end{figure}

Note also that $L_1,L_2,...,L_s$ of $F$ corresponds to $C_1,C_2,...,C_s$ (see Def.\ref{rowgraph}) in $R(G,F,\alpha)$.
Moreover, observe that $R(G,F,\alpha)_C$ is eulerian since it is isomorphic to $G_F$.

Let us call two $3$-row graphs $R_1$, $R_2$ with $s$ columns \textit{identical} if they satisfy $v_{ab}v_{cd} \in E(R_1)$ if and only if $v_{ab}v_{cd} \in E(R_2)$ for every $a,c \in \{1,2,3\}$ and for every $b,d \in \{1,2,...,s\}$.
Recall that for the definition of a $3$-row graph, a labeling of the components of $F$ is needed.

\begin{lemma}\label{identical}
For every rearrangement $R'$ of $R(G,F,\alpha)$, there is a relabeling of the components of $F$ and a perfect coloring $\alpha_1$ of $F$ such that $R(G,F,\alpha_1)$ is identical with $R'$. 
\end{lemma}

\begin{proof}
Let $R(G,F,\alpha)$ have $s$ columns.
Permuting columns corresponds to a relabeling of $L_1,L_2,...,L_s$ and is tantamount to a permutation of $\{1,2,...,s\}$.
Permuting $v_{im}$ and $v_{jm}$ in the $m$-th column corresponds to permuting colors $i$ and $j$ in the perfect coloring $\alpha$ of the corresponding K-component $L_m$. One defines the recoloring $\alpha_1$ with respect to a C-component, analogously. \qed
\end{proof}

\begin{thm}\label{t:swedish}
Let $G$ be a cubic graph with a Kotzig-frame $F$ and a perfect coloring $\alpha$ of $F$. Then $G$ has a $6$-cycle double cover 
if $R(G,F,\alpha)$ has an amiable coloring.
\end{thm}

\begin{proof}
Abbreviate $R(G,F,\alpha)$ by $R$. Let $(f,g)$ be an amiable coloring of $R$ (see Def.\ref{d:swedish}). 
Suppose without loss of generality that 
$$f(v_{ij})=i\,\,\,\,\,\text{for every} \,\,v_{ij} \in V(R) \text{,}\,\,i=1,2,3 \,\,\text{and} \,\,j=1,2,...,s\,\,\,\,\,\,\,\,\,\,\,\,\,(*)$$
  
Note: if need be, one can permute vertices of individual columns of $R$ to obtain this property which would imply a change of the perfect coloring of $F$ (see Lemma \ref{identical}). 
For subsequent considerations, it is useful to regard the coloring $g$ of $E(R)$ as an edge coloring of 
$E(G)-E(F)-F^*$ (recall that $F^*$ is the set of chords of $F$ - see Def.\ref{kotzig}(c)). 
By the same token, the elements of $E(R)$ carry the same name as the corresponding elements in $E(G)-E(F)-F^*$.
All $2$-valent vertices $u$ of a component of $F$ satisfying $\alpha(u)=i$ correspond by Def.\ref{d:detachment} and by
$(*)$ to a vertex $x \in V(R)$ with $f(x)=i$ and vice versa. Therefore $g(e) \not\in \{\alpha(v),\alpha(w)\}$ for every $e \in E(R)$ with $e$ having endvertices $v$,$w$ in $G$.\\
We define three subgraphs of $G$ by using the following edge sets. For $i=1,2,3$ let $E_i= \{e \in E(G)-E(F)-F^*: g(e)=i\}$, $H_i = \{e \in E(F): \alpha(e) \not=i \}$ and let $F^* = L_1 \cupdot L_2 \cupdot L_3$ be such that $vw \in L_i$ 
implies $\alpha (v) \not=i$, $\alpha (w) \not =i$. 
Note that an edge $vw \in F^*$ satisfying $1=\alpha(v)=\alpha(w)$, say, may be part of $L_2$ or $L_3$. Hence the partition of $F^*$ is not necessarily unique.
The set $H_i$ defines a bicolored cycle in each K-component of $F$ which corresponds to a 
hamilton cycle in the homeomorphic Kotzig graph. Note that any one of these cycles in $G$ is incident with an even number of edges of $E_i$. This follows from Def.\ref{d:swedish}(iii) since at least one of the terms in (iii) vanishes.
Moreover, $H_i$ covers all edges of every C-component of $F$ 
which does not contain color $i$ in the coloring $\alpha$. By Def.\ref{d:swedish}, the number of edges of $E_i$ 
incident with a $C$-component is also even.
Finally, note that every $e \in L_i$ is a chord of a cycle of $G[H_i]$.

Set $$J_i = H_i \cup E_i \cup L_i, \,i=1,2,3\,.$$
$G[J_i]$ is a graph consisting of a subdivision of a $3$-edge colorable cubic graph and disjoint cycles. 
These are cycles (of components of $F$) which are not incident with edges of $E_i \cup L_i$. 
Obviously, every even $2$-factor (i.e every cycle of the $2$-factor has even length) of a bridgless cubic graph has a $2$-cycle cover such that the $2$-factor is covered once and the remaining $1$-factor is covered twice. Define $S_i$ to be a $2$-cycle cover of $G[J_i]$, covering every edge of $E_i \cup L_i$ twice and every edge of $H_i$ once. Since $J_1 \cup J_2 \cup J_3=E(G)$, $S=S_1 \cup S_2 \cup S_3$ is a $6$-cycle cover of $G$. It is straightforward to verify that $S$ is a $6$-CDC of $G$. \qed
\end{proof}

\section{Switching colors and $2$-row graphs}

Let $G$ be a graph with a given $2$-edge-coloring using the colors red and blue. $G$ is allowed to have loops. 
A \textit{switch} in a vertex $v \in V(G)$ is a recoloring of the edges incident with $v$ which are not loops by coloring every edge which is colored blue by red and every edge which is colored red by blue, i.e the colors incident with $v$ are switched. A loop incident with $v$ retains its color after a switch.

\begin{thm}\label{t:switch}
Let $G$ be a graph with a given $2$-edge-coloring $f$ using the colors red and blue.
Let $G_b$ be the graph obtained from $G$ by contracting every blue edge, i.e $G_b$ has only red edges.
Then there is a sequence of switches in vertices of $G$ such that the resulting graph has only blue edges if and only if $G_b$ is bipartite.
\end{thm}

\begin{proof}
Let $f^*$ be the $2$-edge-coloring of $G$ which is obtained after having applied a sequence of switchings $S$ in the 
coloring $f$ of $G$.
Define $\beta: V(G) \mapsto \mathbb Z_2$ with $\beta (v)=0$ if the multiplicity of $v$ in $S$ is even, otherwise set 
$\beta (v)=1$. Let $vw \in E(G)$. It is not difficult to see that $f(vw)=f^*(vw)$ iff
$\beta (v) + \beta (w) =0$ and thus $f(vw) \not= f^*(vw)$ iff $\beta (v) + \beta (w) =1$. This observation implies that the $f$-colored graph $G$ can be switched blue
(i.e there exists a sequence of switches such that the resulting graph has only blue edges) iff there is a map $\gamma: V(G) \mapsto \mathbb Z_2$ such that\\
(i) $\gamma (v) + \gamma (w) =0$ if $f(vw)$ is blue and \\
(ii) $\gamma (v) + \gamma (w) =1$ if $f(vw)$ is red. \\
Let us contract one blue edge, $xy$ say, in the $f$-colored graph $G$.
It is straightforward to see that $G/xy$ can be switched blue iff $G$ can be switched blue. By proceeding analogously for the remaining blue edges we obtain that $G$ can be switched blue iff $G_b$ can be switched blue.
By implication (ii) above, it follows that $G_b$ is bipartite iff $G_b$ and thus $G$ can be switched blue which finishes the proof. \qed
\end{proof}

We define a rearrangement for $2$-row graphs which are defined analogously to $3$-row graphs. Let $R$ be a $2$-row graph 
with $s$ columns and let $U \subseteq \{1,2,3,...,s\}$, then $R^U$ denotes the $2$-row graph which is obtained from $R$ by renaming $v_{1u}$ by $v_{2u}$ and $v_{2u}$ by $v_{1u}$, for every $u \in U$. 

\begin{cor}\label{cor:switch}
Let $R$ be a $2$-row graph such that $R_C$ is a forest. Then there is a set $U \subseteq \{1,2,3,...,s\}$ such that $R^U$ has no edge which joins a vertex from the first row to a vertex in the second row.
\end{cor}

\begin{proof}
Recall that the edges of $R$ and $R_C$ correspond bijectively to each other.
Using this fact we define a $2$-edge-coloring of $R$ and and thus also of $R_C$: color every edge of $R$ which joins a vertex from the first row to the second row by red and color all remaining edges blue. 
 It is straightforward to see that
 a switch in $c_j \in V(R_C)$ corresponds to a renaming of $v_{1j}$ with $v_{2j}$ in $R$ and vice versa. Hence in order to prove the corollary it suffices to show that $R_C$ allows a sequence of switches such that finally the graph has only blue edges. Since contractions of edges in an acyclic graph result in an acyclic graph, and since acyclic graphs are bipartite, the proof follows by applying Theorem \ref{t:switch}. \qed
\end{proof}

\section{Proof of 
the main result}

The use of $t$-joins will be essential for the proof.
A $t$-join with $t \subseteq V(G)$ is a spanning subgraph $T^*$ of $G$ such that $d(v,T^*)$ is odd if and only if $v \in t$.
Note that $T^*$ may have isolated vertices. The following lemma is well known; we quote it because we will use it several times.

\begin{lemma}\label{lem:tjoin}
Let $G$ be a graph and $t \subseteq V(G)$. Then $G$ has an acyclic $t$-join if and only if the number of vertices of $t$
in every component of $G$ is even.
\end{lemma}


\PROOF{of Theorem \ref{main}}
Let $G_F$ be well connected with respect to a given perfect coloring $\alpha$. Abbreviate $R(G,F,\alpha)$ by $R$ and let $s$ denote the number of components of $F$. Recall that every vertex of $G_F$ and every column $C_j$, $j=1,2,3...,s$, correspond to a component of $F$. Furthermore, every edge of $R$ corresponds to an edge of $E(G)-E(F)-F^*$ and vice versa. The edges of $R$ retain the edge labels of the corresponding edges in $E(G)$.\\

Without loss of generality, we assume the following to hold for $R$:\\
{\bf (i)} the columns $C_1,C_2,...,C_k$ correspond to the K-components of $F$ and the columns $C_{k+1},C_{k+2},...,C_s$ correspond to the $C$-components of $F$.\\
{\bf (ii)} all edges of $H$ have color $1$ (see Def.\ref{wellconnected}).\\
{\bf (iii)} $|V(H)|$ is maximal, i.e the perfect coloring and $H$ is chosen such that $H$ contains a maximum number of $C$-vertices in $G_{F,\alpha}$ (see Def.\ref{wellconnected}).\\
{\bf (iv)} if $H$ contains C-vertices, then $C_{k+1},C_{k+2},...,C_l$, $l\leq s$ are the corresponding columns in $R$. Moreover, 
 $\alpha$ can be assumed such that $v_{1m} \in V(R)$, $m=l+1,l+2,...,s$ is an isolated vertex of $R$, i.e the C-component corresponding to $C_m$ does not have a vertex of $\alpha$-color $1$.\\
  
Recall that $V_i=\{v_{ij}:1\leq j\leq s\} $, $i=1,2,3$. 
By our assumptions all non-isolated vertices in $V_1$ are part of one component 
in $R[V_1]$ which corresponds to $H$.
We will show that $R$ has an amiable coloring which will, by Theorem \ref{t:swedish}, finish the proof. Since $R$ has an amiable coloring trivially if $s=1$, we assume $s>1$.\\


Set $R_{2,3}=R[V_2 \cup V_3]$ and identify $v_{2j}$ with $v_{3j}$ in $R_{2,3}$, $j=1,2,...,s$, 
to obtain the graph $R^*_{2,3}$. Let $Y^*_{2,3}$
denote the set of vertices in $R^*_{2,3}$ with odd degree. Then $R^*_{2,3}$ has an acyclic $t$-join $T^*_{2,3}$ for $t=Y^*_{2,3}$ (if $Y^*_{2,3} = \emptyset$, let $T^*_{2,3}$ be the edgeless graph consisting of all vertices of $R^*_{2,3}$).
Every edge of $T^*_{2,3}$ corresponds to an edge of $R_{2,3}$. 
Denote the edges of $T^*_{2,3}$ corresponding to edges in $R[V_2]$ by $T^*_2$ and denote the edges of $T^*_{2,3}$ corresponding to edges in $R[V_3]$ by $T^*_3$. 
Obviously, $T^*_{2,3}$ may contain an edge corresponding to an edge joining a vertex in $V_2$ with a vertex in $V_3$.
However, we can assume the opposite for $R$:\\
{\bf (v)} No edge of $T^*_{2,3}$ corresponds to an edge in $R_{2,3}$ which joins a vertex of $V_2$ with a vertex of $V_3$.\\
This assumption is possible since otherwise we could apply Corollary \ref{cor:switch} to the $2$-row graph contained in $R_{2,3}$ which yields $T^*_{2,3}$ by vertex identification. Note that this would mean a modification of $\alpha$ by exchanging the $\alpha$-color $2$ and $3$ in some components of $F$ (which does not change $H$). This would imply a rearrangement of $R$ which is no problem by Lemma \ref{arrange}.

We define the vertex coloring of an amiable coloring of $R$ (see Def.\ref{d:swedish}) as follows:\\ 
$$f(v_{1j})=2, f(v_{2j})=3,f(v_{3j})=1, \,j=1,2,...,s\,\,.$$
Moreover, we use (v) and set 
$$g(e)=2 \,\,\text {iff} \,\,e \in E(R_{2,3})- E(T^*_2) - E(T^*_3)\,\,.$$
Thus, (iii) in Def.\ref{d:swedish} is fulfilled for $i=2$. 
To obtain the edges of $R$ with $g$-color $1$, we first set
$$R^*_{1,2}=(R[V_1 \cup V_2]-E(R[V_2])) \cup T^*_2$$ and 
$$Y_1= \{v_{1m} \in V_1: d(v_{1m},R^*_{1,2}) \not\equiv d(v_{2m},R^*_{1,2}) \pod{2} \,\,\,\text{with} \,\,\,m \in \{1,2,...,s\}\}\,\,.$$

We claim that $$v_{1m'} \not\in Y_1, \forall m' \in \{l+1,l+2,...,s\}\,.$$
First we note that by (iv), $v_{1m'}$ is an isolated vertex in $R$.
If $d(v_{2m'},R^*_{1,2})=0$, we are finished. Hence we assume $d(v_{2m'},R^*_{1,2}) \not=0$.
Then $d(v_{2m'},R)$ is even since $d(v_{1m'},R)=d(v_{3m'},R)=0$. Moreover no vertex in $V_1$ is adjacent to $v_{2m'}$, otherwise it contradicts assumption (iii). Hence, $v_{2m'}$ and $v_{3m'}$ identified,
form in $R^*_{2,3}$ a vertex of even degree. Therefore, $d(v_{2m'},T^*_2)$ is even since $d(v_{3m'},R)=0$. Thus $d(v_{2m'},R^*_{1,2})$ is even, implying 
$d(v_{1m'},R^*_{1,2}) \equiv  d(v_{2m'},R^*_{1,2}) \pod{2}$, implying $v_{1m'} \not\in Y_1$.\\

By Lemma \ref{lem:tjoin} and since $Y_1$ is contained in the vertex set of 
the component in $R[V_1]$ corresponding to $H$ (see the above claim), $R[V_1]$ has a $t$-join $T^*_1$ for $t=Y_1$. Now, set $$g(e)=1 \,\,\text {iff} \,\,e \in E(R^*_{1,2})-E(T^*_1)$$ and $$g(e)=3 \,\,\text {iff} \,\, e \in E(R)-g^{-1}(2)-g^{-1}(1)\,\,.$$ Since $R_C$ is eulerian and since (iii) in Def.\ref{d:swedish} is fulfilled for $i=1,2$, it also holds for $i=3$ which finishes the proof.
\qed

By the proof of the above theorem and the properties of the $3$-row graph used in the proof, we obtain the following corollary which also implies Corollary \ref{cor:rows2} and Corollary \ref{cor:rows3}.

\begin{cor}\label{cor:rows}
Let $R$ be a $3$-row graph with $R_C$ being eulerian. Suppose that for some $i \in \{1,2,3\}$, at most one component of $R[V_i]$ has more than one vertex. Moreover, suppose that every column of $R$ which contains an isolated vertex of $R[V_i]$, has the following property: there is at most one vertex $v$ in the column such that $d(v,R) \not= 0$. Then $R$ has an amiable coloring.
\end{cor}

\begin{cor}\label{cor:rows2}
Let $R$ be a $3$-row graph such that $R_C$ is eulerian. Then $R$ has an amiable coloring if $R[V_i]$ is connected for some $i \in \{1,2,3\}$.
\end{cor}

\begin{cor}\label{cor:rows3}
Let $R$ be a $3$-row graph with $s$ columns and with $R_C$ being eulerian. Let $\{p,q\} \subseteq \{1,2,...,s\}$ and suppose that every column $C_j$ of $R$ with $j \not\in \{p,q\}$ contains two isolated vertices in $R$. Then $R$ has an amiable coloring.
\end{cor}

\begin{proof}
Suppose there is a $vw$-path in $R$ with $v \in C_p$ and $w \in C_q$. Then we choose the shortest one and call it $P$. Let $R'$ be a rearrangement of $R$ such that all vertices of $P$ are in the first row of $R'$. Denote by $V'_1$ the vertex set of the first row of $R'$. Moreover, let $R$ have been rearranged in a manner such that $R'[V(P)]$ is the only component of $R'[V'_1]$ having more than one vertex. Then Corollary \ref{cor:rows} proves the result. Now suppose that a $vw$-path in $R$ does not exist. Then $R$ can be regarded as two disjoint $3$-row graphs where one contains the vertices of $C_p$ and where the other one contains the vertices of $C_q$. Both $3$-row graphs can be rearranged such that Corollary \ref{cor:rows} can be applied which finishes the proof.   \qed
\end{proof}

Note that Corollary \ref{cor:rows3} and Theorem \ref{t:swedish} implies Corollary \ref{mainimply1}.

\section{A reformulation of amiable colorings}

Let $R$ be a $3$-row graph with a fixed vertex $3$-vertex-coloring $f$. Is there a $3$-edge-coloring $g$ of $R$ such 
that $(f,g)$ is an amiable coloring? The theorem below, gives some answer to this problem implying a reformulation of an amiable coloring (see Corollary \ref{cor:zebra}). For this purpose we need some additional terminology.


\begin{definition}\label{d:zebra}
A parity-coloring of a $3$-row graph $R$ with $s$ columns is a vertex $2$-coloring $\phi$ of $R$ using the colors black and white such that for $j=1,2,...,s$ all of the following hold \\
(i) $\phi(v_{1j}) = \phi(v_{2j})$ if and only if $ d(v_{1j},R[V_1 \cup V_2]) \equiv |N(v_{2j}) \cap V_1| \pod{2} $.\\
(ii) $\phi(v_{2j}) = \phi(v_{3j})$ if and only if $d(v_{2j},R[V_2 \cup V_3]) \equiv d(v_{3j},R[V_2 \cup V_3]) \pod{2}$.\\
(iii) The number of black vertices is even in every component of $R[V_i]$ for $i=1,2,3$.
\end{definition}

The idea of a parity-coloring is used in the proof of Theorem \ref{main}. Note that vertices of different rows have to satisfy different conditions in Def.\ref{d:zebra}. To obtain more symmetry, we introduce the following coloring.

\begin{definition}\label{d:orca}
A symmetric parity-coloring of a $3$-row graph $R$ with $s$ columns is a vertex $2$-coloring $\phi$ of $R$ using the colors black and white such that for $j=1,2,...,s$ all of the following hold\\ 
(i) $\phi(v_{1j}) = \phi(v_{2j})$ if and only if $d(v_{1j},R[V_1 \cup V_{2}]) \equiv  |N(v_{2j}) \cap V_{1}|  \pod{2} $.\\
(ii) $\phi(v_{2j}) = \phi(v_{3j})$ if and only if $d(v_{2j},R[V_2 \cup V_{3}]) \equiv  |N(v_{3j}) \cap V_{2}|  \pod{2} $.\\
(iii) The number of black vertices is even in every component of $R[V_i]$ for $i=1,2,3$.
\end{definition}


\begin{thm}\label{t:extension}
Let $R$ be a $3$-row graph with $s$ columns and with $R_C$ being eulerian. Let $f$ be a vertex $3$-coloring of $R$ such that 
$f(v_{ij})=i$ for $j=1,2,...,s$. Then the following three properties are equivalent: (1) $f$ can be extended to an amiable coloring $(f,g)$ of $R$; (2) $R$ has a parity-coloring; (3) $R$ has a symmetric parity-coloring.
\end{thm}


\begin{proof}
Addition involving indices which denote the three rows of $R$ is regarded modulo $3$, i.e $3+1=1$.\\
We show that (1) is equivalent to (3). We first prove that the edge coloring $g$ implies a symmetric parity-coloring $\phi$ of $R$.
For $i=1,2,3$, set $$\phi(v_{ij})= \,\text{black iff} \,\,\, d_{i+1}(v_{ij},R[V_i]) \,\,\text{is odd,}\,\, j=1,2,...,s\,\,.$$ 
Consider the case $i=1$ and thus the subgraph of $R[V_1]$ induced by the edges of color $2$ in $R[V_1]$. Then each component of this subgraph is contained in a component of $R[V_1]$. Thus the number of vertices satisfying $d_2(v, R[V_1]) \equiv 1 \pod{2}$ is even in every component of $R[V_1]$. 
The cases $i=2,3$ are dealt with analogously. Therefore (iii) of Def.\ref{d:orca} is fulfilled. \\
Consider $\phi(v_{1j})$, $\phi(v_{2j})$. To prove (i) of Def.\ref{d:orca}, we show that $$d_2(v_{1j},R[V_1]) + d_3(v_{2j},R[V_2]) \equiv d(v_{1j},R[V_1 \cup V_2]) + |N(v_{2j}) \cap V_{1}| \pod{2}\,\,\,.\,\,
\,\,\,\,\,\,\,\,\,\,(*)$$

We observe that $$d_2(v_{1j},R[V_1]) + d_3(v_{2j},R[V_2])=$$
$$=(d(v_{1j},R[V_1 \cup V_2])- d_3(v_{1j},R[V_1 \cup V_2])) + (d_3(v_{2j},R[V_1 \cup V_2])-|N(v_{2j}) \cap V_{1}|)$$
(note that by definition of an amiable coloring, $d_i(v_{ij},R)=0$ if $f(v_{ij})=i$). 

Moreover $$d_3(v_{1j},R[V_1 \cup V_2])+ d_3(v_{2j},R[V_1 \cup V_2]) \equiv 0 \pod{2}$$ by Def.\ref{d:swedish} (iii). 
The preceding equation and this congruence imply the validity of $(*)$.
Now (i) of Def.\ref{d:orca} follows. Condition (ii) of Def.\ref{d:orca} is shown analogously.\\

We now prove that a symmetric parity-coloring $\phi$ of $R$ implies an amiable coloring $(f,g)$ of $R$.  
First, we show that $$ \phi(v_{3j}) = \phi(v_{1j}) \,\,\text{iff}\,\,
d(v_{3j},R[V_3 \cup V_1]) \equiv  |N(v_{1j}) \cap V_{3}|  \pod{2}\,\,.\,\,\,\,\,\,\,\,\,\,\,\,(\bullet) $$ 

Set $a_i:=d(v_{ij},R[V_i \cup V_{i+1}]) + |N(v_{i+1\,j}) \cap V_{i}| $ for $i=1,2,3$. 
We must thus prove that $a_3$ is even iff 
$\phi(v_{3j})=\phi(v_{1j})$.\\ 
Suppose first that $\phi(v_{3j})=\phi(v_{1j})$. Then either $\phi(v_{2j})=\phi(v_{1j})=\phi(v_{3j})$ and thus $a_1$, $a_2$ are both even, or 
$\phi(v_{2j}) \not\in \{\phi(v_{1j}),\phi(v_{3j})\}$ and thus $a_1$, $a_2$ are both odd. 
By definition of $a_i$ and since $R_C$ is eulerian by hypothesis, it follows that
$$d(v_{1j},R)+d(v_{2j},R)+d(v_{3j},R)=a_1+a_2+a_3 \equiv 0 \pod {2} \,\,.\,\,\,\,\,\,\,\,\,\,\,\,\,\,\,\,\,\,\,\,(**)$$ 
It now follows that $a_3$ is even. If $\phi(v_{3j}) \not= \phi(v_{1j})$, then $a_1 + a_2 \equiv 1 \pod{2}$, and thus $a_3$ is odd by $(**)$.\\
 Conversely suppose $a_3$ is even. By $(**)$, $a_3$ is even iff $a_1+a_2$ is even.
Then either $\phi(v_{1j}) \not= \phi(v_{2j})$ and $\phi(v_{2j}) \not= \phi(v_{3j})$ (if $a_1$, $a_2$ are both odd) or 
$\phi(v_{1j}) = \phi(v_{2j})$ and $\phi(v_{2j}) = \phi(v_{3j})$ (if $a_1$, $a_2$ are both even), implying 
$\phi(v_{1j}) = \phi(v_{3j})$. Thus $(\bullet)$ has been proved.

Let $T_i$ be a $t$-join of $R[V_i]$ where $t$ is the set of the black vertices of $R[V_i]$, $i=1,2,3$ (see 
Def.\ref{d:orca} (iii)). Moreover,
set $g(e)=i$ iff $e \in E(R[V_{i+1} \cup V_{i+2}])-E(T_{i+1})-(E(R[V_{i+2}])-E(T_{i+2}))$.
It follows that for every $e \in E(R[V_i])$, $g(e)=i+1$ iff $e \in T_i$. 
Note that (i),(ii) of Def.\ref{d:swedish} are fulfilled: (i) holds by hypothesis and (ii) holds by definition of $g$.
We show that (iii) in Def.\ref{d:swedish} for $i=3$ is fulfilled. The cases, $i=1,2$ are proven analogously. We must show that $A:=d_3(v_{1j},R[V_1 \cup V_2]) + d_3(v_{2j},R[V_1 \cup V_2])$ is even.
Clearly, $$A=|N(v_{2j}) \cap V_{1}| + d(v_{1j},R[V_1 \cup V_{2}]) - d(v_{1j}, T_1) + d(v_{2j}, T_2)\,.$$

Since $d(v_{1j}, T_1) + d(v_{2j}, T_2)$ is even iff $\phi(v_{1j})=\phi(v_{2j})$ 
by definition of $T_1$, $T_2$ and since by Def.\ref{d:orca} (i)
$$|N(v_{2j}) \cap V_{1}| + d(v_{1j},R[V_1 \cup V_{2}]) \,\,\text{is even iff} \,\,\phi(v_{1j})=\phi(v_{2j})\,\,,$$ $A$
is even which finishes the first part of the proof. \\

We show that (1) is equivalent to (2). We prove first that a parity-coloring $\phi$ of $R$ implies an amiable coloring $(f,g)$ of $R$; $f$ is defined by hypothesis and satisfies Def.\ref{d:swedish} (i). Let $T^*_i$ be a $t$-join of $R[V_i]$ where $t$ comprises the black vertices of $V_i$, $i=1,2,3$. 
For $e \in E(R[V_2])$, set $g(e)=3$ iff $e \in T^*_2$ and for $e \in E(R[V_3])$, set $g(e)=2$ iff $e \in T^*_3$.
Moreover, set $g(e)=1$ iff $e \in E(R[V_2 \cup V_3]) - (E(T^*_2) \cup E(T^*_3))$. 
Obviously, $$d_1(v_{2j},R)+d_1(v_{3j},R)=$$ 
 $$=d(v_{2j},R[V_2 \cup V_3])- d(v_{2j},R[T^*_2]) + d(v_{3j},R[V_2 \cup V_3]) - d(v_{3j},R[T^*_3])\,\,\,.$$ 
By definition of $T^*_i$, $d(v_{ij},R[T^*_i]) \equiv 1 \pod{2}$ iff $\phi(v_{ij})=$ black, $i=2,3$. Therefore and by (ii) of Def.\ref{d:zebra}, for both cases $\phi(v_{2j})=\phi(v_{3j})$ and $\phi(v_{2j})\not=\phi(v_{3j})$, the above sum is even which proves (iii) of Def.\ref{d:swedish} for $i=1$.\\
We extend the coloring $g$ by setting 
$$g(e)=3 \,\,\text{iff} \,\, e \in E(R[V_1 \cup V_2]) - E(T^*_1)-(E(R[V_2])-E(T^*_2))\,\,.$$ 
For $e \in E(R[V_1])$, set $g(e)=2$ iff $e \in T^*_1$.  
The color $\phi(v_{2j})$ ($\phi(v_{1j})$) decides the parity of $d(v_{2j},R[T^*_2])$ ($d(v_{1j},R[T^*_1])$).
Therefore and since $$d_3(v_{1j},R)+d_3(v_{2j},R)=$$ 
$$=d(v_{1j},R[V_1 \cup V_2])- d(v_{1j},R[T^*_1]) + d(v_{2j},R[T^*_2]) + |N(v_{2j}) \cap V_1|$$ 
and since (i) of Def.\ref{d:zebra} holds, this sum is even which verifies (iii) of Def.\ref{d:swedish} for $i=3$. Finally set $g(e)=2$ iff $e \in E(R)-g^{-1}(1)-g^{-1}(3)$. Then Def.\ref{d:swedish} (ii) is fulfilled. Since 
Def.\ref{d:swedish} (iii) holds for $i=1,3$, it also holds for $i=2$. Thus $(f,g)$ is an amiable coloring of $R$. \\
Suppose finally that an amiable coloring $(f,g)$ is given. Set $\phi(v_{2j})=$ black iff $d_3(v_{2j},R[V_2])$ is odd, $\phi(v_{3j})=$ black iff $d_2(v_{3j},R[V_3])$ is odd, and $\phi(v_{1j})=$ black iff $d_2(v_{1j},R[V_1])$ is odd.
By the same arguments which we used above to prove that (iii) of Def.\ref{d:orca} holds, it follows that (iii) of 
Def.\ref{d:zebra} is fulfilled. By Def.\ref{d:swedish} (iii), 
$$d(v_{2j},R[V_2 \cup V_3] - d_3(v_{2j},R[V_2])+ d(v_{3j},R[V_2 \cup V_3])- d_2(v_{3j},R[V_3])$$ 
is even since this sum equals $d_1(v_{2j},R)+d_1(v_{3j},R)$. Hence (ii) of Def.\ref{d:zebra} is fulfilled. Finally, since $d_3(v_{1j},R) + d_3(v_{2j},R)$ is even by Def.\ref{d:swedish} (iii) and equals $$d(v_{1j},R[V_1 \cup V_2])- d_2(v_{1j},R[V_1]) + d_3(v_{2j},R[V_2])+|N(v_{2j}) \cap V_1|\,\,,$$ 
Def.\ref{d:zebra} (i) is fulfilled. The theorem now follows. \qed 
\end{proof}

\begin{cor}\label{cor:zebra}
Let $R$ be a $3$-row graph such that $R_C$ is eulerian. Then the following three properties are equivalent: (i) $R$ has an amiable coloring; (ii) a rearrangement of $R$ has a parity-coloring; (iii) a rearrangement of $R$ has a 
symmetric parity-coloring .
\end{cor}

\section*{Acknowledgments}
A.Hoffmann-Ostenhof was supported by the Austrian Science Fund (FWF) project P 26686.
H.Fleischner and R.H\"aggkvist were supported by the Austrian Science Fund (FWF) project P 27615.
Moreover, R.H\"aggkvist was also supported by the Department of Mathematics and Mathematical 
Statistics, Ume{\aa} University, Sweden.


\end{document}